\documentclass[11pt]{amsart}
\usepackage{amssymb}
\usepackage{amsmath}
\usepackage[active]{srcltx}
\usepackage{t1enc}
\usepackage[latin2]{inputenc}
\usepackage{verbatim}
\usepackage{amsmath,amsfonts,amssymb,amsthm}
\usepackage[mathcal]{eucal}
\usepackage{enumerate}
\usepackage[centertags]{amsmath}
\usepackage{graphics}

\newtheorem{theorem}{Theorem}

\newtheorem{lemma}{Lemma}

\newtheorem{corollary}{Corollary}

\begin{document}
\author{L. E. Persson and G. Tephnadze}
\title[Fejér means]{A sharp boundedness result concerning some maximal
operators of Vilenkin-Fejér means }
\address{L.-E. Persson, Department of Engineering Sciences and Mathematics,
Luleä\ University of Technology, SE-971 87 Luleä, Sweden and Narvik
University College, P.O. Box 385, N-8505, Narvik, Norway.}
\email{larserik@ltu.se}
\address{G. Tephnadze, Department of Mathematics, Faculty of Exact and
Natural Sciences, Tbilisi State University, Chavchavadze str. 1, Tbilisi
0128, Georgia and Department of Engineering Sciences and Mathematics, Luleä\
University of Technology, SE-971 87 Luleä, Sweden.}
\email{giorgitephnadze@gmail.com}
\thanks{The research was supported by a Swedish Institute scholarship,
provided within the framework of the SI Baltic Sea Region Cooperation/Visby
Programme.}
\date{}
\maketitle

\begin{abstract}
In this paper we derive the maximal subspace of positive numbers, for which
the restricted maximal operator of Fejér means in this subspace is bounded
from the Hardy space $H_{p}$ to the space $L_{p}$ for all $0<p\leq 1/2.$
Moreover, we prove that the result is in a sense sharp.
\end{abstract}

\date{}

\textbf{2000 Mathematics Subject Classification.} 42C10, 42B25.

\textbf{Key words and phrases:} Vilenkin system, Vilenkin group, Fejér
means, martingale Hardy space, maximal operator, Vilenkin-Fourier series.

\section{Introduction}

\bigskip In the one-dimensional case the weak (1,1)-type inequality for the
maximal operator of Fejér means 
\begin{equation*}
\sigma ^{\ast }f:=\sup_{n\in \mathbb{N}}\left\vert \sigma _{n}f\right\vert
\end{equation*}%
can be found in Schipp \cite{Sc} for Walsh series and in Pál, Simon \cite{PS}
for bounded Vilenkin series. Fujji \cite{Fu} and Simon \cite{Si2} verified
that $\sigma ^{\ast }$ is bounded from $H_{1}$ to $L_{1}$. Weisz \cite{We2}
generalized this result and proved boundedness of $\sigma ^{\ast }$ from the
martingale space $H_{p}$ to the Lebesgue space $L_{p}$ for $p>1/2$. Simon 
\cite{Si1} gave a counterexample, which shows that boundedness does not hold
for $0<p<1/2.$ A counterexample for $p=1/2$ was given by Goginava \cite%
{GoAMH} (see also \cite{BGG} and \cite{BGG2}). Weisz \cite{We4} proved that
the maximal operator of the Fejér means $\sigma ^{\ast }$ is bounded from
the Hardy space $H_{1/2}$ to the space $weak-L_{1/2}$. The boundedness of
weighted maximal operators are considered in \cite{GoSzeged}, \cite{GNCz}, 
\cite{tep2}, \cite{tep3}. Weisz \cite{We3}(see also \cite{We1}) also proved
that the following theorem is true:

\textbf{Theorem W:} \textbf{(Weisz)} Let $p>0.$ Then the maximal operator 
\begin{equation*}
\sigma ^{\nabla ,\ast }f=\underset{n\in \mathbb{N}}{\sup }\left\vert \sigma
_{M_{n}}f\right\vert
\end{equation*}%
\ is bounded from the Hardy space $H_{p}$ to the space $L_{p}.$

The main aim of this paper is to generalize Theorem W and find the maximal
subspace of positive numbers, for which the restricted maximal operator of
Fejér means in this subspace is bounded from the Hardy space $H_{p}$ to the
space $L_{p}$ for all $0<p\leq 1/2.$ As applications, both some well-known
and new results are pointed out.

This paper is organized as follows: in order not to disturb our discussions
later on some definitions and notations are presented in Section 2. The main
results and some of its consequences can be found in Section 3. For the
proofs of the main results we need some auxiliary Lemmas, some of them are
new and of independent interest. These results are presented in Section 4.
The detailed proofs are given in Section 5.

\section{Definitions and Notations}

Denote by $\mathbb{N}_{+}$ the set of the positive integers, $\mathbb{N}:=%
\mathbb{N}_{+}\cup \{0\}.$ Let $m:=(m_{0},m_{1},\dots )$ be a sequence of
the positive integers not less than 2. Denote by $Z_{m_{n}}:=\{0,1,\ldots
,m_{n}-1\}$ the additive group of integers modulo $m_{n}$. Define the group $%
G_{m}$ as the complete direct product of the groups $Z_{m_{n}}$ with the
product of the discrete topologies of $Z_{m_{n}}`$s. In this paper we
discuss bounded Vilenkin groups, i.e. the case when $\sup_{n\in \mathbb{N}%
}m_{n}<\infty .$

The direct product $\mu $ of the measures $\mu _{n}\left( \{j\}\right)
:=1/m_{n},\ (j\in Z_{m_{n}})$ is the Haar measure on $G_{m}$ with $\mu
\left( G_{m}\right) =1.$

The elements of $G_{m}$ are represented by sequences 
\begin{equation*}
x:=\left( x_{0},x_{1},\ldots,x_{n},\ldots \right),\ \left( x_{n}\in
Z_{m_{n}}\right).
\end{equation*}

It is easy to give a base for the neighbourhood of $G_{m}:$

\begin{equation*}
I_{0}\left( x\right) :=G_{m},\ I_{n}(x):=\{y\in G_{m}\mid
y_{0}=x_{0},\ldots,y_{n-1}=x_{n-1}\}\,\,\left( x\in G_{m},\text{ }n\in 
\mathbb{N}\right).
\end{equation*}

Set $I_{n}:=I_{n}\left( 0\right) ,$ for $n\in \mathbb{N}_{+}$ \ and

\begin{equation*}
e_{n}:=\left( 0,\dots ,0,x_{n}=1,0,\dots \right) \in G_{m}\qquad \left( n\in 
\mathbb{N}\right) .
\end{equation*}

Denote 
\begin{equation*}
I_{N}^{k,l}:=\left\{ 
\begin{array}{l}
\text{ }I_{N}(0,\dots ,0,x_{k}\neq 0,0,\dots ,0,x_{l}\neq 0,x_{l+1,\dots ,%
\text{ }}x_{N-1\text{ }}),\text{ \ \ }k<l<N, \\ 
\text{ }I_{N}(0,\dots ,0,x_{k}\neq 0,0,\dots ,0),\text{ \qquad }l=N.%
\end{array}%
\text{ }\right.
\end{equation*}

It is easy to show that 
\begin{equation}
\overline{I_{N}}=\left( \overset{N-2}{\underset{i=0}{\bigcup }}\overset{N-1}{%
\underset{j=i+1}{\bigcup }}I_{N}^{i,j}\right) \bigcup \left( \underset{i=0}{%
\bigcup\limits^{N-1}}I_{N}^{i,N}\right) .  \label{2}
\end{equation}

\bigskip If we define the so-called generalized number system based on $m$
in the following way : 
\begin{equation*}
M_{0}:=1,\ M_{n+1}:=m_{n}M_{n}\ \ \ (n\in \mathbb{N}),
\end{equation*}%
then every $n\in \mathbb{N}$ can be uniquely expressed as $%
n=\sum_{k=0}^{\infty }n_{k}M_{k},$ where $n_{k}\in Z_{m_{k}}\ (k\in \mathbb{N%
}_{+})$ and only a finite number of $n_{k}`$s differ from zero. Let 
\begin{equation*}
\left\langle n\right\rangle :=\min \{j\in \mathbb{N}:n_{j}\neq 0\}\text{ \ \
and \ \ \ }\left\vert n\right\vert :=\max \{j\in \mathbb{N}:n_{j}\neq 0\},
\end{equation*}%
that is $M_{\left\vert n\right\vert }\leq n\leq M_{\left\vert n\right\vert
+1}.$ Set $d\left( n\right) =\left\vert n\right\vert -\left\langle
n\right\rangle ,$ \ for \ all \ \ $n\in \mathbb{N}.$

Next, we introduce on $G_{m}$ an orthonormal system, which is called the
Vilenkin system. At first, we define the complex-valued function $%
r_{k}\left( x\right) :G_{m}\rightarrow \mathbb{C},$ the generalized
Rademacher functions, by 
\begin{equation*}
r_{k}\left( x\right) :=\exp \left( 2\pi ix_{k}/m_{k}\right) ,\text{ }\left(
i^{2}=-1,x\in G_{m},\text{ }k\in \mathbb{N}\right) .
\end{equation*}

Now, define the Vilenkin system $\,\,\,\psi :=(\psi _{n}:n\in\mathbb{N})$ on 
$G_{m}$ as: 
\begin{equation*}
\psi _{n}(x):=\prod\limits_{k=0}^{\infty }r_{k}^{n_{k}}\left( x\right) \ \ \
\left( n\in\mathbb{N}\right).
\end{equation*}

Specifically, we call this system the Walsh-Paley system, when $m\equiv 2.$

The norms (or quasi-norms) of the spaces $L_{p}(G_{m})$ and $%
weak-L_{p}\left( G_{m}\right) $ $\ \left( 0<p<\infty \right) $ are
respectively defined by 
\begin{equation*}
\left\Vert f\right\Vert _{p}^{p}:=\int_{G_{m}}\left\vert f\right\vert
^{p}d\mu ,\text{ \ \ }\left\Vert f\right\Vert _{weak-L_{p}}^{p}:=\underset{%
\lambda >0}{\sup }\,\lambda ^{p}\mu \left( f>\lambda \right) <\infty .
\end{equation*}%
\qquad

The Vilenkin system is orthonormal and complete in $L_{2}\left( G_{m}\right) 
$ (see \cite{Vi}).

If $\ f\in L_{1}\left( G_{m}\right) $ we can define Fourier coefficients,
partial sums, Dirichlet kernels, Fejér means, Fejér kernels with respect to
the Vilenkin system in the usual manner: 
\begin{equation*}
\widehat{f}\left( k\right) :=\int_{G_{m}}f\overline{\psi }_{k}d\mu \text{%
\thinspace\ \ \ }\left( \text{ }k\in \mathbb{N}\right) ,
\end{equation*}%
\begin{eqnarray*}
S_{n}f &:&=\sum_{k=0}^{n-1}\widehat{f}\left( k\right) \psi _{k},\ \text{%
\qquad }D_{n}:=\sum_{k=0}^{n-1}\psi _{k\text{ }}\text{ \ \ \ \ \ \ }\left( 
\text{ }n\in \mathbb{N}_{+}\text{ }\right) , \\
\sigma _{n}f &:&=\frac{1}{n}\sum_{k=0}^{n-1}S_{k}f,\text{\qquad\ \ \ }K_{n}:=%
\frac{1}{n}\overset{n-1}{\underset{k=0}{\sum }}D_{k}\text{ \ \ }\left( \text{
}n\in \mathbb{N}_{+}\text{ }\right) .
\end{eqnarray*}

Recall that (see e.g. \cite{AVD}) 
\begin{equation}
\quad \hspace*{0in}D_{M_{n}}\left( x\right) =\left\{ 
\begin{array}{l}
\text{ }M_{n},\text{ \thinspace \thinspace if \thinspace \thinspace }x\in
I_{n}, \\ 
\text{ }0,\text{ \ \ \thinspace \thinspace \thinspace if \ \thinspace
\thinspace }x\notin I_{n},%
\end{array}%
\right.  \label{3}
\end{equation}%
and \ 
\begin{equation}
D_{s_{n}M_{n}}=D_{s_{n}M_{n}}\sum_{k=0}^{s_{n}-1}\psi
_{kM_{n}}=D_{M_{n}}\sum_{k=0}^{s_{n}-1}r_{n}^{k},  \label{9dn}
\end{equation}%
where $n\in \mathbb{N}$ and $1\leq s_{n}\leq m_{n}-1.$

The $\sigma $-algebra generated by the intervals $\left\{ I_{n}\left(
x\right) :x\in G_{m}\right\} $ will be denoted by $\digamma _{n}$ $\left(
n\in \mathbb{N}\right) .$ Denote by $f=\left( f^{\left( n\right) },n\in 
\mathbb{N}\right) $ a martingale with respect to $\digamma _{n}$ $\left(
n\in \mathbb{N}\right) $ (for details see e.g. \cite{We1}). The maximal
function of a martingale $f$ is defined by \qquad 
\begin{equation*}
f^{\ast }=\sup_{n\in \mathbb{N}}\left\vert f^{\left( n\right) }\right\vert .
\end{equation*}

In the case $f\in L_{1}(G_{m}),$ the maximal functions are also be given by 
\begin{equation*}
f^{\ast }\left( x\right) =\sup_{n\in \mathbb{N}}\frac{1}{\left\vert
I_{n}\left( x\right) \right\vert }\left\vert \int_{I_{n}\left( x\right)
}f\left( u\right) \mu \left( u\right) \right\vert .
\end{equation*}

For $0<p<\infty $ the Hardy martingale spaces $H_{p}\left( G_{m}\right) $
consist of all martingales $f$, for which 
\begin{equation*}
\left\Vert f\right\Vert _{H_{p}}:=\left\Vert f^{\ast }\right\Vert
_{p}<\infty .
\end{equation*}

If $f\in L_{1}(G_{m}),$ then it is easy to show that the sequence $\left(
S_{M_{n}}\left( f\right) :n\in \mathbb{N}\right) $ is a martingale. If $%
f=\left( f^{\left( n\right) },n\in \mathbb{N}\right) $ is martingale, then
the Vilenkin-Fourier coefficients must be defined in a slightly different
manner: $\qquad \qquad $ 
\begin{equation*}
\widehat{f}\left( i\right) :=\lim_{k\rightarrow \infty
}\int_{G_{m}}f^{\left( k\right) }\left( x\right) \overline{\psi }_{i}\left(
x\right) d\mu \left( x\right) .
\end{equation*}

The Vilenkin-Fourier coefficients of $f\in L_{1}\left( G_{m}\right) $ are
the same as those of the martingale $\left( S_{M_{n}}\left( f\right) :n\in 
\mathbb{N}\right) $ obtained from $f$.

A bounded measurable function $a$ is said to be a p-atom if there exists an
interval $I$, such that%
\begin{equation*}
\int_{I}ad\mu =0,\text{ \ \ }\left\Vert a\right\Vert _{\infty }\leq \mu
\left( I\right) ^{-1/p},\text{ \ \ supp}\left( a\right) \subset I.\qquad
\end{equation*}%
\qquad

\section{The Main Result and applications}

Our main result reads:

\begin{theorem}
\label{theorem0fejermax}a)Let $0<p\leq 1/2$ and $\left\{ n_{k}:k\geq
0\right\} $ be a sequence of positive numbers, such that%
\begin{equation*}
\sup_{k}\rho \left( n_{k}\right) \leq c<\infty .
\end{equation*}%
Then the maximal operator 
\begin{equation*}
\widetilde{\sigma }^{\ast ,\nabla }f=\underset{k\in \mathbb{N}}{\sup }%
\left\vert \sigma _{n_{k}}f\right\vert
\end{equation*}%
is bounded from the Hardy space $H_{p}$ to the space $L_{p}.$

The statement in a) is sharp in the following sense:

b) Let $0<p<1/2$ and $\left\{ n_{k}:k\geq 0\right\} $ be a sequence of
positive numbers, such that%
\begin{equation}
\sup_{k}\rho \left( n_{k}\right) =\infty .  \label{10}
\end{equation}%
Then there exists a martingale $f\in H_{p}$ such that 
\begin{equation*}
\underset{k\in \mathbb{N}}{\sup }\left\Vert \sigma _{n_{k}}f\right\Vert
_{p}=\infty .
\end{equation*}
\end{theorem}

As a first application we obtain the previous mentioned result by Weisz \cite%
{We1}, \cite{We3} (Theorem W).

\begin{corollary}
Let $p>0$ and $f\in H_{p}.$ Then the maximal operator $\sigma ^{\nabla ,\ast
}f$ \ is bounded from the Hardy space $H_{p}$ to the space $L_{p}.$
\end{corollary}

Moreover, we get the following new information:

\begin{corollary}
Let $0<p<1/2,$ $f\in H_{p}$ and $\left\{ n_{k}:k\geq 0\right\} $ be any
sequence of positive numbers. Then the maximal operator 
\begin{equation*}
\widetilde{\sigma }^{\ast ,\nabla }f=\underset{k\in \mathbb{N}}{\sup }%
\left\vert \sigma _{n_{k}}f\right\vert
\end{equation*}%
is bounded from the Hardy space $H_{p}$ to the space $L_{p}$ if and only if 
\begin{equation*}
\sup_{k}\rho \left( n_{k}\right) <\infty .
\end{equation*}
\end{corollary}

\begin{corollary}
Let $0<p<1/2,$ $f\in H_{p}$ and $\left\{ n_{k}:k\geq 0\right\} $ be any
sequence of positive numbers. Then $\sigma _{n_{k}}f$ are uniformly bounded
from the Hardy space $H_{p}$ to the space $L_{p}$ if and only if 
\begin{equation*}
\sup_{k}\rho \left( n_{k}\right) <\infty .
\end{equation*}
\end{corollary}

\section{AUXILIARY LEMMAS}

For the proof of Theorem 1 we need the following Lemmas:

\begin{lemma}[see e.g. \protect\cite{We3}]
\label{lemma1} A martingale $f=\left( f^{\left( n\right) },n\in \mathbb{N}%
\right) $ is in $H_{p}\left( 0<p\leq 1\right) $ if and only if there exist a
sequence $\left( a_{k},k\in \mathbb{N}\right) $ of p-atoms and a sequence $%
\left( \mu _{k},k\in \mathbb{N}\right) $ of real numbers such that for every 
$n\in \mathbb{N}:$%
\begin{equation}
\qquad \sum_{k=0}^{\infty }\mu _{k}S_{M_{n}}a_{k}=f^{\left( n\right) }
\label{1}
\end{equation}%
and%
\begin{equation*}
\qquad \sum_{k=0}^{\infty }\left\vert \mu _{k}\right\vert ^{p}<\infty .
\end{equation*}%
Moreover, $\left\Vert f\right\Vert _{H_{p}}\backsim \inf \left(
\sum_{k=0}^{\infty }\left\vert \mu _{k}\right\vert ^{p}\right) ^{1/p},$
where the infimum is taken over all decomposition of $f$ of the form (\ref{1}%
).
\end{lemma}

\begin{lemma}[see e.g. \protect\cite{We3}]
\label{lemma2} Suppose that an operator $T$ is $\sigma $-linear and for some 
$0<p\leq 1$%
\begin{equation*}
\int\limits_{\overset{-}{I}}\left\vert Ta\right\vert ^{p}d\mu \leq
c_{p}<\infty ,
\end{equation*}%
for every $p$-atom $a$, where $I$ denote the support of the atom. If $T$ is
bounded from $L_{\infty \text{ }}$ to $L_{\infty },$ then 
\begin{equation*}
\left\Vert Tf\right\Vert _{p}\leq c_{p}\left\Vert f\right\Vert _{H_{p}}.
\end{equation*}
\end{lemma}

\begin{lemma}[see \protect\cite{gat}]
\label{lemma3} Let $n>t,$ $t,n\in \mathbb{N},$ $x\in I_{t}\backslash $ $%
I_{t+1}$. Then 
\begin{equation*}
K_{M_{n}}\left( x\right) =\left\{ 
\begin{array}{ll}
0, & \text{if }x-x_{t}e_{t}\notin I_{n}, \\ 
\frac{M_{t}}{1-r_{t}\left( x\right) }, & \text{if }x-x_{t}e_{t}\in I_{n}.%
\end{array}%
\right.
\end{equation*}
\end{lemma}

\begin{lemma}[see \protect\cite{tep3}]
\label{lemma5b}\ Let $x\in I_{N}^{i,j},$ $i=0,\dots ,N-1,$ $j=i+1,\dots ,N$.
Then%
\begin{equation*}
\int_{I_{N}}\left\vert K_{n}\left( x-t\right) \right\vert d\mu \left(
t\right) \leq \frac{cM_{i}M_{j}}{M_{N}^{2}}.
\end{equation*}
\end{lemma}

We also need the next two new Lemmas of independent interest:

\begin{lemma}
\label{lemma7}Let $t,s_{n},n\in \mathbb{N},$ and $1\leq s_{n}\leq m_{n}-1$.
Then\bigskip 
\begin{equation*}
s_{n}M_{n}K_{s_{n}M_{n}}=\sum_{l=0}^{s_{n}-1}\left(
\sum_{i=0}^{l-1}r_{n}^{i}\right) M_{n}D_{M_{n}}+\left(
\sum_{l=0}^{s_{n}-1}r_{n}^{l}\right) M_{n}K_{M_{n}}
\end{equation*}%
and%
\begin{equation*}
\left\vert K_{s_{n}M_{n}}\left( x\right) \right\vert \geq \frac{M_{n}}{\sqrt{%
2}\pi s_{n}},\text{ for }x\in I_{n+1}\left( e_{n-1}+e_{n}\right) .
\end{equation*}%
Moreover, if $x\in I_{t}/I_{t+1},$ \ $x-x_{t}e_{t}\notin I_{n}$ and $n>t,$
then 
\begin{equation}
K_{s_{n}M_{n}}(x)=0.  \label{100kn}
\end{equation}
\end{lemma}

\begin{proof}
We can write that%
\begin{eqnarray}
s_{n}M_{n}K_{s_{n}M_{n}} &=&\sum_{l=0}^{s_{n}-1}\sum_{k=lM_{n}}^{\left(
l+1\right) M_{n}-1}D_{k}  \label{skn} \\
&=&\sum_{l=0}^{s_{n}-1}\sum_{k=lM_{n}}^{\left( l+1\right)
M_{n}-1}D_{k}=\sum_{l=0}^{s_{n}-1}\sum_{k=0}^{M_{n}-1}D_{k+lM_{n}}.  \notag
\end{eqnarray}

Let $0\leq k<M_{n}$. Since 
\begin{equation}
D_{j+lM_{n}}=D_{lM_{n}}+\psi _{M_{n}}^{l}D_{j}=D_{lM_{n}}+r_{n}^{l}D_{j},%
\text{ when }\,\,j<lM_{n}  \label{8k}
\end{equation}%
if we apply (\ref{9dn}) we obtain that\qquad\ 
\begin{eqnarray*}
D_{k+lM_{n}} &=&\sum_{m=0}^{lM_{n}-1}\psi
_{m}+\sum_{m=lM_{n}}^{lM_{n}+k-1}\psi _{m} \\
&=&D_{lM_{n}}+\sum_{m=0}^{k-1}\psi
_{m+lM_{n}}=D_{lM_{n}}+r_{n}^{l}\sum_{m=0}^{k-1}\psi _{m} \\
&=&\left( \sum_{s=0}^{l-1}r_{n}^{s}\right) D_{M_{n}}+r_{n}^{l}D_{k}.
\end{eqnarray*}

By applying (\ref{skn}) we get that 
\begin{eqnarray*}
&&s_{n}M_{n}K_{s_{n}M_{n}}=\sum_{l=0}^{s_{n}-1}%
\sum_{k=0}^{M_{n}-1}D_{k+lM_{n}} \\
&=&\sum_{l=0}^{s_{n}-1}\sum_{k=0}^{M_{n}-1}\left( \left(
\sum_{i=0}^{l-1}r_{n}^{i}\right) D_{M_{n}}+r_{n}^{l}D_{k}\right) \\
&=&\sum_{l=0}^{s_{n}-1}\left( \sum_{i=0}^{l-1}r_{n}^{i}\right)
M_{n}D_{M_{n}}+\sum_{l=0}^{s_{n}-1}r_{n}^{l}\sum_{k=0}^{M_{n}-1}D_{k} \\
&=&\sum_{l=0}^{s_{n}-1}\left( \sum_{i=0}^{l-1}r_{n}^{i}\right)
M_{n}D_{M_{n}}+\sum_{l=0}^{s_{n}-1}r_{n}^{l}M_{n}K_{M_{n}}.
\end{eqnarray*}

Let $x\in I_{n+1}\left( e_{n-1}+e_{n}\right) .$ By Lemma \ref{lemma3} we
have that%
\begin{equation}
\left\vert K_{M_{n}}\left( x\right) \right\vert =\frac{M_{n-1}}{\left\vert
1-r_{n-1}\left( x\right) \right\vert }=\frac{M_{n-1}}{\sqrt{2}\sin \pi
/m_{n-1}}.  \label{11k}
\end{equation}%
Moreover, since%
\begin{equation*}
\sum_{u=1}^{s_{n}-1}r_{n}^{u}\left( x\right) =\sum_{u=1}^{s_{n}-1}\cos
u+\sum_{u=1}^{s_{n}-1}\sin u
\end{equation*}%
\begin{equation*}
=\frac{\cos \pi s_{n}/m_{n}\sin \pi \left( s_{n}-1\right) /m_{n}}{\sin \pi
/m_{n}}i+\frac{\sin \pi s_{n}/m_{n}\sin \pi \left( s_{n}-1\right) /m_{n}}{%
\sin \pi /m_{n}},
\end{equation*}%
it follows that%
\begin{equation}
\left\vert \sum_{u=1}^{s_{n}-1}r_{n}^{u}\left( x\right) \right\vert =\frac{%
\sin \pi \left( s_{n}-1\right) /m_{n}}{\sin \pi /m_{n}}\geq 1.  \label{12k}
\end{equation}

By combining (\ref{3}), (\ref{12k}) and the first part of Lemma \ref{lemma7}
we immediately get that%
\begin{eqnarray*}
&&\left\vert s_{n}M_{n}K_{s_{n}M_{n}}\left( x\right) \right\vert =\left\vert
\left( \sum_{l=0}^{s_{n}-1}r_{n}^{l}\left( x\right) \right)
M_{n}K_{M_{n}}\left( x\right) \right\vert \\
&=&\frac{M_{n}M_{n-1}}{\sqrt{2}\sin \pi /m_{n-1}}\geq \frac{%
M_{n}M_{n-1}m_{n-1}}{\sqrt{2}\pi }\geq \frac{M_{n}^{2}}{\sqrt{2}\pi }.
\end{eqnarray*}

Now, let $t,s_{n},n\in \mathbb{N},\ n>t,\ x\in I_{t}\backslash I_{t+1}$. If $%
x-x_{t}e_{t}\notin I_{n}$, then, by combining (\ref{3}), (\ref{9dn}) and
Lemma \ref{lemma3}, we obtain that 
\begin{equation*}
D_{M_{n}}(x)=K_{M_{n}}(x)=0.
\end{equation*}%
By again using the first part of Lemma \ref{lemma7} we get that%
\begin{equation*}
K_{s_{n}M_{n}}(x)=0.
\end{equation*}

The proof is complete.\vspace{0pt}\qquad
\end{proof}

\begin{lemma}
\label{lemma8}Let $n\in \mathbb{N}.$ Then%
\begin{equation}
\left\vert K_{n}\left( x\right) \right\vert \leq \frac{c}{n}%
\sum_{l=\left\langle n\right\rangle }^{\left\vert n\right\vert
}M_{l}\left\vert K_{M_{l}}\right\vert \leq c\sum_{l=\left\langle
n\right\rangle }^{\left\vert n\right\vert }\left\vert K_{M_{l}}\right\vert
\label{10k}
\end{equation}%
and 
\begin{equation}
\left\vert nK_{n}\right\vert \geq \frac{M_{\left\langle n\right\rangle }^{2}%
}{\sqrt{2}\pi \lambda },\text{ \ \ \ }x\in I_{\left\langle n\right\rangle
+1}\left( e_{\left\langle n\right\rangle -1}+e_{\left\langle n\right\rangle
}\right) ,  \label{9k}
\end{equation}%
where $\lambda :=\sup m_{n}.$
\end{lemma}

\begin{proof}
It is well-known that (see \cite{bt}) 
\begin{eqnarray*}
&&nK_{n}=\sum_{k=1}^{r}\left( \prod_{j=1}^{k-1}r_{n_{j}}^{s_{j}}\right)
s_{k}M_{n_{k}}K_{s_{k}M_{n_{k}}} \\
&&+\sum_{k=1}^{r-1}\left( \prod_{j=1}^{k-1}r_{n_{j}}^{s_{j}}\right)
n^{(k)}D_{s_{k}M_{n_{k}}}.
\end{eqnarray*}%
Hence the proof follows by just combining (\ref{3}) and (\ref{9dn}) with
Lemmas \ref{lemma3} and \ref{lemma7}.
\end{proof}

\section{Proof of Theorem 1}

\begin{proof}[ Proof of Theorem \protect\ref{theorem0fejermax}]
Since 
\begin{equation}
\sup_{n}\int_{G_{m}}\left\vert K_{n}\left( x\right) \right\vert d\mu \left(
x\right) \leq c<\infty  \label{4}
\end{equation}%
\vspace{0pt}we obtain that $\widetilde{\sigma }^{\ast ,\vartriangle }$ is
bounded from $L_{\infty }$ to $L_{\infty }.$ According to Lemma \ref{lemma2}
we find that the proof of Theorem \ref{theorem0fejermax} will be complete,
if we show that%
\begin{equation*}
\int_{\overline{I_{N}}}\left\vert \widetilde{\sigma }^{\ast ,\vartriangle
}a\left( x\right) \right\vert <c<\infty ,
\end{equation*}%
for every $p$-atom $a,$ with support$\ I$ and $\mu \left( I\right)
=M_{N}^{-1}.$ We may assume that $I=I_{N}.$ It is easy to see that $\sigma
_{n_{k}}\left( a\right) =0$ when $n_{k}\leq M_{N}.$ Therefore, we can
suppose that $n_{k}>M_{N}$.

Since $\left\Vert a\right\Vert _{\infty }\leq M_{N}^{1/p}$ we find that%
\begin{eqnarray}
&&\left\vert \sigma _{n_{k}}a\left( x\right) \right\vert \leq
\int_{I_{N}}\left\vert a\left( t\right) \right\vert \left\vert
K_{n_{k}}\left( x-t\right) \right\vert d\mu \left( t\right)  \label{400} \\
&\leq &\left\Vert a\right\Vert _{\infty }\int_{I_{N}}\left\vert
K_{n_{k}}\left( x-t\right) \right\vert d\mu \left( t\right) \leq
M_{N}^{1/p}\int_{I_{N}}\left\vert K_{n_{k}}\left( x-t\right) \right\vert
d\mu \left( t\right) .  \notag
\end{eqnarray}

Without lost the generality we may assume that $i<j$. Let $x\in I_{N}^{i,j}\ 
$and $j<\left\langle n_{k}\right\rangle .$ Then $x-t\in I_{N}^{i,j}$ for $%
t\in I_{N}\ \ $and according to Lemma \ref{lemma3}, we obtain that 
\begin{equation*}
\left\vert K_{M_{l}}\left( x-t\right) \right\vert =0,\text{ \ for all }%
\left\langle n_{k}\right\rangle \leq \text{ }l\leq \left\vert
n_{k}\right\vert .
\end{equation*}

By applying (\ref{400}) and (\ref{10k}) in Lemma \ref{lemma8}, we get that 
\begin{eqnarray}
\left\vert \sigma _{n_{k}}a\left( x\right) \right\vert  &\leq &M_{N}^{1/p}%
\overset{\left\vert n_{k}\right\vert }{\underset{l=\left\langle
n_{k}\right\rangle }{\sum }}\int_{I_{N}}\left\vert K_{M_{l}}\left(
x-t\right) \right\vert d\mu \left( t\right) =0,\text{ }  \label{401} \\
\text{for }x &\in &I_{N}^{i,j},\text{ \ }0\leq i<j<\left\langle
n_{k}\right\rangle \leq l\leq \left\vert n_{k}\right\vert .  \notag
\end{eqnarray}

Let $x\in I_{N}^{i,j},\,$where $\left\langle n_{k}\right\rangle \leq j\leq N.
$ Then, in the view of Lemma \ref{lemma5b}, we have that 
\begin{equation*}
\int_{I_{N}}\left\vert K_{n_{k}}\left( x-t\right) \right\vert d\mu \left(
t\right) \leq \frac{cM_{i}M_{j}}{M_{N}^{2}}.
\end{equation*}

By using again (\ref{400}) we find that 
\begin{equation}
\left\vert \sigma _{n_{k}}a\left( x\right) \right\vert \leq
c_{p}M_{N}^{1/p-2}M_{i}M_{j}.  \label{403}
\end{equation}

Since $n_{k}\geq M_{N}$ we obtain that $\left\vert n_{k}\right\vert \geq N$
and%
\begin{equation*}
\sup_{k}\left( N-\left\langle n_{k}\right\rangle \right) \leq \sup_{k}\left(
\left\vert n_{k}\right\vert -\left\langle n_{k}\right\rangle \right) \leq
\sup_{k}\rho \left( n_{k}\right) <c<\infty.
\end{equation*}%
Thus,%
\begin{equation}
\frac{M_{N}^{1-p}}{M_{\left\langle n_{k}\right\rangle }^{1-p}}\leq \frac{%
M_{\left\vert n_{k}\right\vert }^{1-p}}{M_{\left\langle n_{k}\right\rangle
}^{1-p}}\leq \lambda ^{\left( \left\vert n_{k}\right\vert -\left\langle
n_{k}\right\rangle \right) \left( 1-p\right) }=\lambda ^{\rho \left(
n_{k}\right) \left( 1-p\right) }<c<\infty ,  \label{399}
\end{equation}%
where \ \ $\lambda =\sup_{k}m_{k}.$

By combining (\ref{2}) and (\ref{400})-(\ref{399}) we get that 
\begin{eqnarray*}
&&\int_{\overline{I_{N}}}\left\vert \widetilde{\sigma }^{\ast ,\vartriangle
}a\right\vert ^{p}d\mu \\
&=&\overset{N-2}{\underset{i=0}{\sum }}\overset{N-1}{\underset{j=i+1}{\sum }}%
\sum\limits_{x_{s}=0,s\in
\{j+1,...,N-1\}}^{m_{s}-1}\int_{I_{N}^{i,j}}\left\vert \widetilde{\sigma }%
^{\ast ,\vartriangle }a\right\vert ^{p}d\mu +\overset{N-1}{\underset{i=0}{%
\sum }}\int_{I_{N}^{k,N}}\left\vert \widetilde{\sigma }^{\ast ,\vartriangle
}a\right\vert ^{p}d\mu \\
&\leq &\overset{\left\langle n_{k}\right\rangle -1}{\underset{i=0}{\sum }}%
\overset{N-1}{\underset{j=\left\langle n_{k}\right\rangle }{\sum }}%
\sum\limits_{x_{s}=0,s\in
\{j+1,...,N-1\}}^{m_{s}-1}\int_{I_{N}^{i,j}}\left\vert \widetilde{\sigma }%
^{\ast ,\vartriangle }a\right\vert ^{p}d\mu \\
&&+\overset{N-2}{\underset{i=\left\langle n_{k}\right\rangle }{\sum }}%
\overset{N-1}{\underset{j=i+1}{\sum }}\sum\limits_{x_{s}=0,s\in
\{j+1,...,N-1\}}^{m_{s}-1}\int_{I_{N}^{i,j}}\left\vert \widetilde{\sigma }%
^{\ast ,\vartriangle }a\right\vert ^{p}d\mu +\overset{N-1}{\underset{i=0}{%
\sum }}\int_{I_{N}^{i,N}}\left\vert \widetilde{\sigma }^{\ast ,\vartriangle
}a\right\vert ^{p}d\mu \\
&\leq &c_{p}M_{N}^{1-2p}\overset{\left\langle n_{k}\right\rangle }{\underset{%
i=0}{\sum }}M_{i}^{p}\overset{N-1}{\underset{j=\left\langle
n_{k}\right\rangle +1}{\sum }}\frac{1}{M_{j}^{1-p}}+M_{N}^{1-2p}\overset{N-2}%
{\underset{i=\left\langle n_{k}\right\rangle }{\sum }}M_{i}^{p}\overset{N-1}{%
\underset{j=i+1}{\sum }}\frac{1}{M_{j}^{1-p}}+c_{p}\overset{N-1}{\underset{%
i=0}{\sum }}\frac{M_{i}^{p}}{M_{N}^{p}} \\
&\leq &\frac{c_{p}M_{N}^{1-2p}}{M_{\left\langle n_{k}\right\rangle }^{1-2p}}%
+c_{p}\leq \frac{c_{p}M_{\left\vert n_{k}\right\vert }^{1-p}}{%
M_{\left\langle n_{k}\right\rangle }^{1-p}}+c_{p}\leq c_{p}\lambda ^{\left(
\left\vert n_{k}\right\vert -\left\langle n_{k}\right\rangle \right) \left(
1-p\right) }<\infty .
\end{eqnarray*}

The proof of the a) part is complete.

Now, we prove the b) part of Theorem \ref{theorem0fejermax}. \textbf{\ }Let $%
\left\{ n_{k}:k\geq 0\right\} $ be a sequence of positive numbers,
satisfying condition (\ref{10}). Then 
\begin{equation}
\sup_{k}\frac{M_{\left\vert n_{k}\right\vert }}{M_{\left\langle
n_{k}\right\rangle }}=\infty .  \label{12h}
\end{equation}

Under condition (\ref{12h}) there exists a sequence $\left\{ \alpha _{k}:%
\text{ }k\geq 0\right\} \subset \left\{ n_{k}:\text{ }k\geq 0\right\} $ such
that $\alpha _{0}\geq 3$ and 
\begin{equation}
\sum_{\eta =0}^{\infty }\frac{M_{\left\langle \alpha _{k}\right\rangle
}^{\left( 1-2p\right) /2}}{M_{\left\vert \alpha _{k}\right\vert }^{\left(
1-2p\right) /2}}<c<\infty .  \label{12hh}
\end{equation}

Let \qquad 
\begin{equation*}
f^{\left( n\right) }=\sum_{\left\{ k;\text{ }\left\vert \alpha
_{k}\right\vert <n\right\} }\lambda _{k}a_{k},
\end{equation*}%
where 
\begin{equation*}
\lambda _{k}=\frac{\lambda M_{\left\langle \alpha _{k}\right\rangle
}^{\left( 1/p-2\right) /2}}{M_{\left\vert \alpha _{k}\right\vert }^{\left(
1/p-2\right) /2}}
\end{equation*}%
and%
\begin{equation*}
a_{k}=\frac{M_{\left\vert \alpha _{k}\right\vert }^{1/p-1}}{\lambda }\left(
D_{M_{\left\vert \alpha _{k}\right\vert +1}}-D_{M_{\left\vert \alpha
_{k}\right\vert }}\right) .
\end{equation*}

B applying Lemma \ref{lemma1} we can conclude that $f\in H_{p}.$

It is easy to show that%
\begin{equation}
\widehat{f}(j)=\left\{ 
\begin{array}{l}
M_{\left\vert \alpha _{k}\right\vert }^{1/2p}M_{\left\langle \alpha
_{k}\right\rangle }^{\left( 1/p-2\right) /2},\,\,\text{ } \\ 
\text{if \thinspace \thinspace }j\in \left\{ M_{\left\vert \alpha
_{k}\right\vert },...,\text{ ~}M_{\left\vert \alpha _{k}\right\vert
+1}-1\right\} ,\text{ }k=0,1,2..., \\ 
0\text{ },\text{ \thinspace \qquad \thinspace\ \ \ \ \thinspace\ \ \ \ \ }
\\ 
\text{\ if \thinspace \thinspace \thinspace }j\notin
\bigcup\limits_{k=0}^{\infty }\left\{ M_{\left\vert \alpha _{k}\right\vert
},...,\text{ ~}M_{\left\vert \alpha _{k}\right\vert +1}-1\right\} .\text{ }%
\end{array}%
\right.  \label{6aacharp}
\end{equation}

Moreover,%
\begin{equation*}
\sigma _{_{\alpha _{k}}}f=\frac{1}{\alpha _{k}}\sum_{j=1}^{M_{\left\vert
\alpha _{k}\right\vert }}S_{j}f+\frac{1}{\alpha _{k}}\sum_{j=M_{\left\vert
\alpha _{k}\right\vert }+1}^{\alpha _{k}}S_{j}f:=I+II.
\end{equation*}%
Let $M_{\left\vert \alpha _{k}\right\vert }<j\leq \alpha _{k}.$ Then, by
applying (\ref{6aacharp}), we get that 
\begin{equation}
S_{j}f=S_{M_{\left\vert \alpha _{k}\right\vert }}f+M_{\left\vert \alpha
_{k}\right\vert }^{1/2p}M_{\left\langle \alpha _{k}\right\rangle }^{\left(
1/p-2\right) /2}\left( D_{_{j}}-D_{M_{\left\vert \alpha _{k}\right\vert
}}\right) .  \label{8aafn}
\end{equation}

By using (\ref{8aafn}) we can write $II$ as%
\begin{eqnarray*}
II &=&\frac{\alpha _{k}-M_{\left\vert \alpha _{k}\right\vert }}{\alpha _{k}}%
S_{M_{\left\vert \alpha _{k}\right\vert }}f+\frac{M_{\left\vert \alpha
_{k}\right\vert }^{1/2p}M_{\left\langle \alpha _{k}\right\rangle }^{\left(
1/p-2\right) /2}}{\alpha _{k}}\sum_{j=M_{\left\vert \alpha _{k}\right\vert
}}^{\alpha _{k}}\left( D_{_{j}}-D_{M_{\left\vert \alpha _{k}\right\vert
}}\right) \\
&:&=II_{1}+II_{2}.
\end{eqnarray*}

It is easy to show that%
\begin{equation*}
\left\Vert II_{1}\right\Vert _{weak-L_{p}}^{p}\leq \left( \frac{\alpha
_{k}-M_{\left\vert \alpha _{k}\right\vert }}{\alpha _{k}}\right)
^{p}\left\Vert S_{M_{\left\vert \alpha _{k}\right\vert }}f\right\Vert
_{weak-L_{p}}^{p}\leq c_{p}\left\Vert f\right\Vert _{H_{p}}^{p}<\infty .
\end{equation*}

By using part a) of Theorem \ref{theorem0fejermax} for the estimation of $I$
we have that 
\begin{equation*}
\left\Vert I\right\Vert _{weak-L_{p}}^{p}=\left( \frac{M_{\left\vert \alpha
_{k}\right\vert }}{\alpha _{k}}\right) ^{p}\left\Vert \sigma _{M_{\left\vert
\alpha _{k}\right\vert }}f\right\Vert _{weak-L_{p}}^{p}\leq c_{p}\left\Vert
f\right\Vert _{H_{p}}^{p}<\infty .
\end{equation*}

Let $x\in $ $I_{_{\left\langle \alpha _{k}\right\rangle +1}}^{\left\langle
\alpha _{k}\right\rangle -1,\left\langle \alpha _{k}\right\rangle }.$ Under
condition (\ref{10}) we can conclude that $\left\langle \alpha
_{k}\right\rangle \neq \left\vert \alpha _{k}\right\vert $ and $\left\langle
\alpha _{k}-M_{\left\vert \alpha _{k}\right\vert }\right\rangle
=\left\langle \alpha _{k}\right\rangle .$ If we apply equality (\ref{8k})
for $l=1$ and estimate (\ref{9k}) in Lemma \ref{lemma8} for $II_{2}$ we
obtain that 
\begin{eqnarray*}
\left\vert II_{2}\right\vert &=&\frac{M_{\left\vert \alpha _{k}\right\vert
}^{1/2p}M_{\left\langle \alpha _{k}\right\rangle }^{\left( 1/p-2\right) /2}}{%
\alpha _{k}}\left\vert \sum_{j=1}^{\alpha _{k}-M_{\left\vert \alpha
_{k}\right\vert }}\left( D_{j+M_{\left\vert \alpha _{k}\right\vert
}}-D_{M_{\left\vert \alpha _{k}\right\vert }}\right) \right\vert \\
&=&\frac{M_{\left\vert \alpha _{k}\right\vert }^{1/2p}M_{\left\langle \alpha
_{k}\right\rangle }^{\left( 1/p-2\right) /2}}{\alpha _{k}}\left\vert \psi
_{M_{\left\vert \alpha _{k}\right\vert }}\sum_{j=1}^{\alpha
_{k}-M_{\left\vert \alpha _{k}\right\vert }}D_{j}\right\vert \\
&\geq &c_{p}M_{\left\vert \alpha _{k}\right\vert }^{1/2p-1}M_{\left\langle
\alpha _{k}\right\rangle }^{\left( 1/p-2\right) /2}\left( \alpha
_{k}-M_{\left\vert \alpha _{k}\right\vert }\right) \left\vert K_{\alpha
_{k}-M_{\left\vert \alpha _{k}\right\vert }}\right\vert \\
&\geq &c_{p}M_{\left\vert \alpha _{k}\right\vert }^{1/2p-1}M_{\left\langle
\alpha _{k}\right\rangle }^{\left( 1/p+2\right) /2}.
\end{eqnarray*}

It follows that%
\begin{eqnarray*}
&&\left\Vert II_{2}\right\Vert _{weak-L_{p}}^{p} \\
&\geq &c_{p}\left( M_{\left\vert \alpha _{k}\right\vert }^{\left(
1/p-2\right) /2}M_{\left\langle \alpha _{k}\right\rangle }^{\left(
1/p+2\right) /2}\right) ^{p}\mu \left\{ x\in G_{m}:\text{ }\left\vert
IV_{2}\right\vert \geq c_{p}M_{\left\vert \alpha _{k}\right\vert }^{\left(
1/p-2\right) /2}M_{\left\langle \alpha _{k}\right\rangle }^{\left(
1/p+2\right) /2}\right\} \\
&\geq &c_{p}M_{\left\vert \alpha _{k}\right\vert }^{1/2-p}M_{\left\langle
\alpha _{k}\right\rangle }^{1/2+p}\mu \left\{ I_{_{\left\langle \alpha
_{k}\right\rangle +1}}^{\left\langle \alpha _{k}\right\rangle
-1,\left\langle \alpha _{k}\right\rangle }\right\} \geq \frac{%
c_{p}M_{\left\vert \alpha _{k}\right\vert }^{1/2-p}}{M_{\left\langle \alpha
_{k}\right\rangle }^{1/2-p}}.
\end{eqnarray*}%
Hence, for large $k$, 
\begin{eqnarray*}
&&\left\Vert \sigma _{\alpha _{k}}f\right\Vert _{weak-L_{p}}^{p} \\
&\geq &\left\Vert II_{2}\right\Vert _{weak-L_{p}}^{p}-\left\Vert
II_{1}\right\Vert _{weak-L_{p}}^{p}-\left\Vert I\right\Vert _{weak-L_{p}}^{p}
\\
&\geq &\frac{1}{2}\left\Vert II_{2}\right\Vert _{weak-L_{p}}^{p}\geq \frac{%
c_{p}M_{\left\vert \alpha _{k}\right\vert }^{1/2-p}}{2M_{\left\langle \alpha
_{k}\right\rangle }^{1/2-p}}\rightarrow \infty ,\text{ as }k\rightarrow
\infty .
\end{eqnarray*}

The proof is complete.
\end{proof}

\end{document}